\documentclass[12pt]{amsart}
\usepackage{amscd}     
\usepackage{amssymb}
\usepackage{amsmath, amsthm, graphics, dsfont}
\usepackage{xypic}     
\LaTeXdiagrams        
\usepackage[all]{xy}
\xyoption{2cell} \UseAllTwocells \xyoption{frame} \CompileMatrices
\allowdisplaybreaks[3]
\usepackage{amsfonts}
\usepackage[normalem]{ulem}
\usepackage{hyperref}
\usepackage{tikz}
\usepackage{mathtools}
\usepackage{verbatim} 
\usepackage{MnSymbol}
\usepackage{latexsym}
\usepackage{epsfig}
\usepackage{enumerate}
\usepackage{times}
\usepackage{xcolor}
\usepackage{geometry}

\newgeometry{vmargin={25mm,25mm}, hmargin={25mm,25mm}}   

\newtheorem{prop}{Proposition}[section]
\newtheorem{lem}[prop]{Lemma}

\newtheorem{thm}[prop]{Theorem}

\numberwithin{equation}{section}

\newcommand{\cO}{\mathcal{O}}

\newcommand{\sL}{\mathcal{L}}
\newcommand{\bbT}{\mathbb{T}}

\date{\today}

\begin{document}

\title{On the derived category of a toric stack bundle}

\author{Qian Chao}
\address{Department of Mathematics\\ Ohio State University\\ 100 Math Tower, 231 West 18th Ave. \\ Columbus,  OH 43210\\ USA}
\email{chao.191@buckeyemail.osu.edu}

\author{Jiun-Cheng Chen}
\address{Department of Mathematics\\ Third General Building\\ National Tsing Hua University\\ No. 101 Sec 2 Kuang Fu Road\\ Hsinchu, 30043\\ Taiwan}
\email{jcchen@math.nthu.edu.tw}

\author[Hsian-Hua Tseng]{Hsian-Hua Tseng}
\address{Department of Mathematics\\ Ohio State University\\ 100 Math Tower, 231 West 18th Ave. \\ Columbus,  OH 43210\\ USA}
\email{hhtseng@math.ohio-state.edu}

\begin{abstract}
We establish some properties of the derived category of torus-equivariant coherent sheaves on a split toric stack bundle. Our main result is a semi-orthogonal decomposition of such a category. 
\end{abstract}

\maketitle

\section{Introduction}

We work over $\mathbb{C}$. Let $B$ be a smooth projective variety. We fix a choice of an ample line bundle $L$ over $B$.

The purpose of this note is to establish some properties of derived category of equivariant coherent sheaves on a split toric stack bundle over $B$. Our main results concern three aspects: 
\begin{enumerate}
    \item spanning classes, see Proposition \ref{spanningclass};
    \item semi-orthogonal decomposition, see Theorem \ref{thm:semi-orth_decomp};
    \item Fourier-Mukai equivalence, see Proposition \ref{prop:FM}.
\end{enumerate}

\subsection{Constructions}
Let $\mathfrak{X}$ be a smooth projective toric stack of positive dimension\footnote{If $\mathfrak{X}$ is $0$-dimensional, then $\mathfrak{X}\simeq BA$ is the classifying stack of a finite abelian group $A$. In this case, toric stack bundles associated to $\mathfrak{X}$ defined in (\ref{eqn:toric_stack_bundle}) below are $A$-gerbes over $B$. Categories of sheaves on these gerbes are well-understood, see e.g. \cite{tt}. }. By construction, $\mathfrak{X}$ is presented as a quotient stack 
\begin{equation}
\mathfrak{X}:=[U_\Sigma/G],    
\end{equation}
where $U_\Sigma\subset \mathbb{C}^n$ is the complement of the union of certain coordinate subspaces, and $G$ acts on $\mathbb{C}^n$ via a homomorhpism $$\alpha:G\to (\mathbb{C}^*)^n.$$ The constructions of $U_\Sigma$ and $\alpha$ are dictated by a combinatorial object called (extended) stacky fan. We refer to \cite{bcs} and \cite{j} for details of the construction of toric stacks. We assume that $\mathfrak{X}$ has trivial generic stabilizer. Hence $\alpha$ is injective.

Put $$\bold{T}:=\text{coker}(\alpha).$$ There is a natural $\bold{T}$-action on $\mathfrak{X}$. Since $\mathfrak{X}$ is projective, $\mathfrak{X}$ contains a $\bold{T}$-fixed point. Also, line bundles on $\mathfrak{X}=[U_\Sigma/G]$ correspond to $G$-equivariant line bundles on $U_\Sigma$, and $\bold{T}$-equivariant line bundles on $\mathfrak{X}$ correspond to $(\mathbb{C}^*)^n$-equivariant line bundles on $U_\Sigma$.

We denote by $D^b_{\bold{T}}(\mathfrak{X})$ the derived category of $\bold{T}$-equivariant coherent sheaves on $\mathfrak{X}$. Note that a $\bold{T}$-equivariant sheaf on $\mathfrak{X}=[U_\Sigma/G]$ is the same as a $(\mathbb{C}^*)^n$-equivariant sheaf on $U_\Sigma$. Hence,
\begin{equation*}
D^b_{\bold{T}}(\mathfrak{X})=D^b_{\bold{T}}([U_\Sigma/G])\simeq D^b_{(\mathbb{C}^*)^n}(U_\Sigma).    
\end{equation*}

We now come to the construction of (split) toric stack bundles. Given a principal $(\mathbb{C}^*)^n$-bundle, $$P\to B,$$ we obtain a toric stack bundle over $B$ with fiber $\mathfrak{X}$,
\begin{equation}\label{eqn:toric_stack_bundle}
\phi: \mathfrak{E}:=[(P\times_{(\mathbb{C}^*)^n}U_\Sigma)/G]\to B,    
\end{equation}
where $G$ acts trivially on $P$ and via $\alpha$ on $U_\Sigma$. More details can be found in \cite{j} and \cite{jt}.

For $b_0\in B$, we denote the fiber of $\mathfrak{E}\to B$ over $b_0$ by $\mathfrak{X}_{b_0}:=\phi^{-1}(b_0)\simeq \mathfrak{X}$.  

\subsection{Organization}
The rest of this note is organized as follows. We begin with some basic materials. In Section \ref{sec:exc_obj} we recall the construction of equivariant exceptional collection on toric stacks in \cite{bo}, see Equation (\ref{eqn:excp_coll}). In Section \ref{sec:spanning} we construct a spanning class for the equivariant derived categories of split toric stack bundles, see Proposition \ref{spanningclass}. 

In Section \ref{sec:semi-ortho} we present our main result. Namely we construct a semi-orthogonal decomposition for the equivariant derived category of a split toric stack bundle, see Theorem \ref{thm:semi-orth_decomp}. This is an extension of the well-known semi-orthogonal decomposition of the derived category of coherent sheaves on a projective space, due to Bondal-Orlov. In Section \ref{sec:FM}, we consider two split toric stack bundles over the same base whose toric fibers are related via a crepant wall-crossing, as considered in  \cite{cij}. In this situation, we prove an equivalence of derived categories between these  toric stack bundles, see Proposition \ref{prop:FM}. 

\subsection{Acknowledgement}
We thank the referee for comments and suggestions. H.-H. T. is supported in part by Simons foundation collaboration grant.

\section{Exceptional objects}\label{sec:exc_obj}
We recall some basic definitions from \cite{bo}. Put $\bbT:=(\mathbb{C}^*)^n$. Consider the stack quotient $[\mathbb{C}^n/\bbT]$. The abelian category of coherent sheaves on this stack is equivalent to $Coh^{\bbT}(\mathbb{C}^n)$, the abelian category of $\bbT$-equivalent coherent sheaves. 
Denote by $$S_n:=\mathbb{C}[z_1,z_2, \cdots,z_n]$$ the polynomial ring of $n$ variables and $[n]=:\{1,2, \cdots , n\}$.
The abelian category $Coh^{\bbT}(\mathbb{C}^n)$ of equivariant sheaves is equivalent to the abelian category $\text{Gr}^{\mathbb{Z}^n}$-$S_n$ of finitely generated $\mathbb{Z}^n$-graded modules over
the polynomial coordinate ring $S_n$ with the 
natural $\mathbb{Z}^n$-grading.
We consider the group $\mathbb{Z}^n$ as  maps from $[n]$ to $\mathbb{Z}$. For every monomial $\prod _{i=1}^{i=n} z_i^{a_i}$, we associate it with an element $\overline{a} : [n] \to  \mathbb{Z}$ such that $\overline{a}(i) = a_i$. 
For quasi-coherent sheaves, we have similar results.
The abelian categories $Qcoh([\mathbb{C}^n/\bbT])$ and $Qcoh^{\bbT}(\mathbb{C}^n)$ of quasi-coherent sheaves are also equivalent to each other and to the abelian category $\text{Gr}^{\mathbb{Z}^n}$-$S_n$ of all $\mathbb{Z}^n$-graded modules. 

For every element $\overline{p} \in  \mathbb{Z}^n$, we define an equivariant line bundle
$$\mathcal{O}(- \overline{p})$$
on the affine space $\mathbb{C}^n$ 
as the coherent sheaf corresponding to the module 
$S_n(- \overline{p})$ which is
the free module $S_n$ with the grading shifted in such a way that the generator of the
module has degree $\overline{p}$. More generally, the twist functor $(\overline{p})$ on the category $\text{Gr}^{\mathbb{Z}^n}$-$S_n$ is defined as follows. It takes a graded module
$M = \bigoplus_{\overline{q} \in \mathbb{Z}^n} M_{\overline{q}}$  
 to the module  $M (\overline{p})$
such that  $M (\overline{p})_{\overline{q}} = M_{\overline{p}+\overline{q}}$ 
and takes a morphism $f : M \to N$ to the same morphism
viewed as a morphism $f (\overline{p}) : M (\overline{p}) \to N (\overline{p})$ between the twisted modules.

Let $I \subset [n]$ be a subset of $[n]$. Let $\overline{p}: [n] \to \mathbb{Z}$ be a function. Denote by 
$$\mathcal{O}_{I, \overline{p}}$$ the $\bbT$-equivariant sheaf on $\mathbb{C}^n$ associated with the $\mathbb{Z}^n$-graded module $$M_{I,\overline{p}}:= \mathbb{C}[z_1,z_2, \cdots, z_n]/\langle z_i, i \in I\rangle (-\overline{p}).$$ 

For each $i \in [n]$, we define $\overline{\epsilon}_{i}(j)= \delta_{ij}$.
For any subset $S \subset [n]$, the  characteristic function is defined as  $\overline{\chi}_S := \sum _{k \in S} \overline{\epsilon}_{k}$.
Following \cite{bo}, we define a lexicographic order as follows:  $(\overline{p}, I) > (\overline{q},J)$ if $\overline{p} > \overline{q}$ or $\overline{p} = \overline{q}$ and $|I| >|J|$. 
We also consider a different system to label these (shifted) $\bbT$-equivariant sheaves. 
For each $(\overline{p}, I)$, we associate an $n$-tuple  $a=(a_1,a_2,\cdots, a_n) \in \mathbb{Z}^n$ as follows: 
set $a_i=\overline{p}(i)$ if $p(i)<0$ or $ i \notin I$ and $a_i= \overline{p}(i)+1$ if $p(i) \geq 0$ and $i \in I$. 
We define 
\begin{equation}\label{eqn:excp_coll}
E_{a}:= \mathcal{O}_{I, \overline{p}} [\bold{p}].
\end{equation}
 
After relabeling, the order is easy to describe:  for $a, b \in \mathbb{Z}^n$, $a >b$ iff $a>b$ under the usual lexicographic order of $\mathbb{Z}^n$. It is more convenient to use this relabeled index system to state the main theorem. However, it is easier to give a proof using the original system.  

The objects $\{E_a\}$ restricts to $\bold{T}$-equivariant objects on $\mathfrak{X}$ via $U_\Sigma\subset \mathbb{C}^n$. We use the same notations for these restrictions. Namely, let 
\begin{equation}\label{eqn:exc_collection}
\{E_a\,\,|\,\, a\in \mathbb{Z}^n\}    
\end{equation}
be the strong full exceptional collection on $D^b_{\bold{T}}(\mathfrak{X})$ obtained in \cite[Theorem 2.7]{bo} and recalled above. 

By construction, objects $E_a\in D^b_{\bold{T}}(\mathfrak{X})$ in (\ref{eqn:exc_collection}) are shifts of pushforwards of certain line bundles on toric strata $[U_a/G]=\mathfrak{X}_a\subset \mathfrak{X}$, i.e. $G$-equivariant line bundles $\mathsf{L}_a$ on $U_a$. Twisting these line bundles by $P\to B$ as in the construction of toric stack bundles gives line bundles $[(P\times_{(\mathbb{C}^*)^n}\mathsf{L}_a)/G]$ on $[(P\times_{(\mathbb{C}^*)^n}U_a)/G]$. Applying the pushforward with respect to the inclusion $$[(P\times_{(\mathbb{C}^*)^n}U_a)/G]\subset [(P\times_{(\mathbb{C}^*)^n}U_\Sigma)/G]$$ and shifts yield objects in $D^b_{\bold{T}}(\mathfrak{E})$ which we denote by $\tilde{E_a}.$ We observe that $\tilde{E_a}$ is flat over $B$.

\section{Spanning classes}\label{sec:spanning}

This Section contains some results on spanning classes. First, we have the following Lemma from \cite[Proposition 1.7]{bo}.

\begin{lem}
The collection of $\bold{T}$-equivariant line bundles
\begin{equation}
\{\sL\in Pic_{\bold{T}}(\mathfrak{X})\}    
\end{equation}
is a spanning class of $D^b_{\bold{T}}(\mathfrak{X})$.
\end{lem}

\begin{proof}
By the definition of spanning class, we need to show that for an object $M$ of $D^b_{\bold{T}}(\mathfrak{X})$, if 
\begin{equation}
{\mathrm Hom}_{D^b_{\bold{T}}(\mathfrak{X})}(\sL, M[r])=0    
\end{equation}
for all $\sL\in Pic_{\bold{T}}(\mathfrak{X})$, $r\in\mathbb{Z}$, then $M=0$.     

We assume that $M$ is a nonzero coherent sheaf. By\footnote{The argument in the proof of \cite[Theorem 4.6]{bh} extends to the $\bold{T}$-equivariant setting.} \cite[Theorem 4.6]{bh}, there is a surjection
\begin{equation}
\bigoplus_i \sL_i\to M.    
\end{equation}
where $\{\mathcal{L}_i\}$ are $\bold{T}$-equivariant line bundles on $\mathfrak{X}$. So for some $i$, there is a nonzero section $\cO_{\mathfrak{X}}\to M\otimes \sL_i^{-1}$. But by assumption $0={\mathrm Hom}_{D^b_{\bold{T}}(\mathfrak{X})}(\sL_i, M[0])={\mathrm Hom}_{D^b_{\bold{T}}(\mathfrak{X})}(\cO_{\mathfrak{X}}, M\otimes \sL^{-1})$. Therefore $M=0$.   
\end{proof}

The following result gives a spanning class for split toric stack bundles. Recall that $L$ is an ample line bundle on $B$.
\begin{prop}\label{spanningclass}
The collection
\begin{equation}\label{eqn:spanning_X}
\{\phi^*L^{\otimes i}\otimes \tilde{\sL}\,\,|\,\, i\in\mathbb{Z}, \sL\in Pic_{\bold{T}}(\mathfrak{X})\}     
\end{equation}
is a spanning class of $D^b_{\bold{T}}(\mathfrak{E})$.
\end{prop}

\begin{proof}
By the definition of spanning class, we need to show that for an object $M$ of $D^b_{\bold{T}}(\mathfrak{E})$, if 
\begin{equation}\label{eqn:spanning_E}
{\mathrm Hom}_{D^b_{\bold{T}}(\mathfrak{E})}(\phi^*L^{\otimes i}\otimes \tilde{\sL}, M[r])=0    
\end{equation}
for all $i\in \mathbb{Z}$, $\sL\in Pic_{\bold{T}}(\mathfrak{X})$, $r\in\mathbb{Z}$, then $M=0$.     

We may assume that $M$ is a coherent sheaf. We have 
\begin{equation*}
{\mathrm Hom}_{D^b_{\bold{T}}(\mathfrak{E})}(\phi^*L^{\otimes i}\otimes \tilde{\sL}, M[r])={\mathrm Hom}_{D^b_{\bold{T}}(\mathfrak{E})}(\cO_{\mathfrak{E}}, \phi^*L^{\otimes -i}\otimes \tilde{\sL}^{-1}\otimes M[r]).    
\end{equation*}
By assumption, we have 
\begin{equation*}
{\mathrm Hom}_{D^b_{\bold{T}}(\mathfrak{E})}(\cO_{\mathfrak{E}}, \tilde{\sL}^{-1}\otimes M[r])=0        
\end{equation*}
for all $\sL\in Pic_{\bold{T}}(\mathfrak{X})$. Since $D^b_{\bold{T}}(\mathfrak{X})$ is generated by $\sL\in Pic_{\bold{T}}(\mathfrak{X})$ \cite[Theorem 4.6]{bh}, it follows that $M|_{\mathfrak{X}_{b_0}}=0$ for any $b_0\in B$. If in addition $\phi_*(M)\neq 0$, then together this implies $M=0$.

Suppose $\phi_*(M)\neq 0$. Let $k\in \mathbb{Z}$ be sufficiently large so that $\phi_*(M)\otimes L^{\otimes k}\neq 0$. We calculate
\begin{equation}
\phi_*(M\otimes \phi^*L^{\otimes k})|_{b_0}=(\phi_*(M)\otimes L^{\otimes k})|_{b_0}=H^0(X_{b_0}, M|_{\mathfrak{X}_{b_0}})=0,     
\end{equation}
which is a contradiction unless $M=0$.
\end{proof}

\section{Semi-orthogonal decomposition}\label{sec:semi-ortho}

\begin{thm}\label{thm:semi-orth_decomp}
The following is a semi-orthogonal decomposition of $D^b_{\bold{T}}(\mathfrak{E})$:
\begin{equation}\label{eqn:SOD_E}
(\phi^*D^b(B)\otimes \tilde{E}_a)_{a\in \mathbb{Z}^m}.    
\end{equation}
\end{thm}

\begin{proof}
Since (\ref{eqn:spanning_E}) is a spanning class of $D^b_{\bold{T}}(\mathfrak{E})$, the subcategory of $D^b_{\bold{T}}(\mathfrak{E})$ generated by (\ref{eqn:SOD_E}) is the whole $D^b_{\bold{T}}(\mathfrak{E})$.

It remains to show semi-orthogonality, namely
\begin{equation}\label{eqn:ortho}
{\mathrm Hom}_{D^b_{\bold{T}}(\mathfrak{E})}(\phi^*M_1\otimes \tilde{E}_a, \phi^*M_2\otimes \tilde{E}_b)=0    
\end{equation}
for $M_1, M_2$ coherent sheaves on $B$ and $a>b$. 

Since $B$ is smooth and projective, every coherent sheaf of $B$ can be resolved by a (finite) complex of direct sums of line bundles. 
It suffices to consider the case where $M_i=L_i, i=1, 2$ are line bundles on $B$. It is also clear that we can assume $M_1$ is the trivial line bundle. Thus we need to compute $${\mathrm Hom}_{D^b_{\bold{T}}(\mathfrak{E})}(\tilde{E}_a, \phi^* L_2 \otimes \tilde{E}_b).$$ We adapt the main idea from \cite{bo}. 
We use $(I, \overline{p})$ (resp. $(J, \overline{q})$) to label the object $\tilde{E}_a$ (resp. $\tilde{E}_b$) as mentioned in the previous section. It is more convenient for our proof. 

As $\mathfrak{E} \to B$ is Zariski locally trivial, we have a Zariski-open cover $\{B_i\}$ of $B$ such that 
$\mathfrak{E}|_{B_i} $ is isomorphic to $[(U_{\Sigma}\times B_i)/G]$; note that $G$ acts trivially on $B_i$. 
Let $j_{\Sigma \times B_i}:  U_{\Sigma_i} \times B_i \to  \mathbb{C}^{n} \times B_i$ be the natural open embedding. 
For each $\tilde{\mathcal{O}}_{I, \overline{p}}$, set $
\tilde{\mathcal{O}}^{U_{\Sigma} \times B_i}_{I,\overline{p}}:=
j^{ *}_{\Sigma \times B_i}(\tilde{\mathcal{O}}_{I, \overline{p}}).$

We first verify that 
\begin{equation}
{\mathrm Hom}_{[U_{\Sigma} \times B_i/G]}(\tilde{\mathcal{O}}^{U_{\Sigma} \times B_i}_{I,\overline{p}}, \phi^* L_2 \otimes \tilde{\mathcal{O}}^{U_{\Sigma} \times B_i}_{J,\overline{q}})^{\bold{T}}=0 
\end{equation} 
when
$(I, \overline{p}) > (J, \overline{q})$. Note again that  $(I, \overline{p}) > (J, \overline{q})$ means that 
$\overline{p} > \overline{q}$  lexicographically or $\overline{p}= \overline{q}$ and $|I| > |J|$.

We resolve the  (shifted) sheaves $\tilde{\mathcal{O}}^{\Sigma \times B_i}_{I,\overline{p}}$ and 
$\tilde{\mathcal{O}}^{\Sigma \times B_i}_{J,\overline{q}}$ by Koszul complexes as in \cite{bo}.  
Note that the argument for the $[(\mathbb{C}^n  \times B_i)/(\mathbb{C}^*)^n]$ case does not directly apply to the general $(\mathbb{C}^*)^n$-invariant $U_{\Sigma} \times B_i $. However, we do still have a local-to-global spectral sequence as we can cover $U_{\Sigma_i} \times B_i$ by $(\mathbb{C}^*)^n$-invariant affine open sets, e.g. $\mathbb{C}^n$ with some unions of  coordinate hyper-planes. 

Consider the $E_2$ terms of the following local-to-global spectral sequence as in \cite[Proposition 2.4]{bo}: 

\begin{equation}
E_2^{r,s}= \bigoplus_{I \setminus (I \cap J) \subset S \subset I, |S|=s} H^{r-s}([(U_{\Sigma} \times B_i)/(\mathbb{C}^*)^n], \phi^* L_2 \otimes \tilde{\mathcal{O}}^{U_{\Sigma} \times B_i}_{I \cup J}(\overline{p}-\overline{q} +\overline{\chi_{S}}) ).
\end{equation}
Following \cite[Lemma 2.3]{bo}, the only possible non-zero terms are when $r=s$. This implies that the spectral sequence degenerates at the $E_2$ page. This implies 
\begin{equation}
\begin{split}
&{\mathrm Ext}^r_{U_{\Sigma} \times B_i/(\mathbb{C}^*)^n} (\mathcal{O}^{U_{\Sigma} \times B_i}_{I, \overline{p}}, \phi^* L_2 \otimes \mathcal{O}^{U_{\Sigma} \times B_i}_{J, \overline{q}})\\
& \\
= &\bigoplus_{I \setminus (I \cap J) \subset S \subset I} H^{r-|S|}([(U_{\Sigma} \times B_i)/(\mathbb{C}^*)^n], \phi^* L_2 \otimes \tilde{\mathcal{O}}^{U_{\Sigma} \times B_i}_{I \cup J}(\overline{p}-\overline{q} +\overline{\chi_{S}}) ).
\end{split}
\end{equation}
Since  we assume $(I,\overline{p}) >(J,\overline{q})$, every term on the right-hand side  is actually zero. 
This shows the left-hand side is zero. Letting $r=0$, we obtain $${\mathrm Hom}_{[U_{\Sigma}/(\mathbb{C}^*)^n\times B_i ]}(\tilde{\mathcal{O}}^{U_{\Sigma} \times B_i}_{I,\overline{p}}, \phi^* L_2 \otimes \tilde{\mathcal{O}}^{U_{\Sigma} \times B_i}_{J,\overline{q}})=0.$$  
The last step is a standard argument. Consider the composition of the global section functor $\Gamma$ and $\mathcal{H}om( \phi^*M_1\otimes \tilde{E}_a, -)$. The Grothendieck spectral sequence for their composition has $E^{p,q}_2 = H ^{p}( \mathfrak{E}, \mathcal{E}xt^q(\phi^*M_1\otimes \tilde{E}_a, \phi^*M_2\otimes \tilde{E}_b))$. We have shown $E^{p,q}_2$=0. This shows (\ref{eqn:ortho}) is valid, as desired. 
\end{proof}

\section{Fourier-Mukai equivalence}\label{sec:FM}
Let $\mathfrak{X}_\pm$ be two smooth toric stacks related by a single crepant wall-crossing, as in \cite{cij}. Another toric stack $\mathfrak{X}$ and two birational morphisms $\pi_\pm: \mathfrak{\tilde{X}}\to \mathfrak{X}_\pm$ are constructed in \cite{cij}. We have $\pi_+^*K_{\mathfrak{X}_+}=\pi_-^*K_{\mathfrak{X}_-}$. Furthermore, the Fourier-Mukai transform $(\pi_+)_*(\pi_-)^*: D^b_T(\mathfrak{X}_-)\to D^b_+(\mathfrak{X}_+)$ is shown to be an equivalence, see \cite{cijs}.

Applying the construction of toric stack bundles to $\mathfrak{\tilde{X}}$ and $\mathfrak{X}_\pm$, we obtain toric stack bundles 
\begin{equation}
\mathfrak{\tilde{E}}\to B, \quad \mathfrak{E}_\pm\to B.    
\end{equation}
The morphisms $\pi_\pm: \mathfrak{\tilde{X}}\to \mathfrak{X}_\pm$ extends to morphisms $\tilde{\pi}_\pm: \mathfrak{\tilde{E}}\to \mathfrak{E}_\pm$. The maps fit into the following diagram
\begin{equation*}
\xymatrix{
 & \mathfrak{\tilde{X}}\subset\mathfrak{\tilde{E}} \ar[dl]_{\tilde{\pi}_-}\ar[dr]^{\tilde{\pi}_+} & \\
 \mathfrak{X}_-\subset \mathfrak{E}_-\ar[dr]_{\phi_-} & & \mathfrak{E}_+\supset \mathfrak{X}_+\ar[dl]^{\phi_+}\\
 & B. &
}
\end{equation*}

\begin{prop}\label{prop:FM}
The Fourier-Mukai transform $\mathbf{FM}=(\tilde{\pi}_+)_*(\tilde{\pi}_-)^*: D^b_{\bold{T}}(\mathfrak{E}_-)\to D^b_{\bold{T}}(\mathfrak{E}_+)$ is an equivalence.    
\end{prop}
\begin{proof}
We calculate 
\begin{equation}
\mathbf{FM}(\phi_-^*M\otimes \tilde{\sL})=\phi_+^*M\otimes\mathbf{FM}(\tilde{\sL}).    
\end{equation}
Let $\{U_i\}$ be an open cover of $B$ such that $\phi_\pm^{-1}(U_i)\simeq U_i\times \mathfrak{X}_\pm$. By the results of \cite{cijs}, the restrictions $\mathbf{FM}|_{\phi_-^{-1}(U_i)}$ are equivalences. 
To get the global case that $\mathbf{FM}$ is an equivalence, we adapt the argument 
in \cite[Proposition 3.2]{c} to the $\bold{T}$-equivariant setting. 
Denote the right adjoint of $\mathbf{FM}$ by $\mathbf{G}$.  
It sufficies to show that $a \cong \mathbf{G} \circ \mathbf{FM} (a)$ for
all $a \in D^b_{\bold{T}} (\mathfrak{E}_- )$ and $\mathbf{FM} \circ \mathbf{G}(b) \cong  b$ for all $b \in D^b_{\bold{T}}(\mathfrak{E}_+)$. 
Since we are in the smooth case, it is not necessary to work with the category $D_{qc}$ as in \cite{c}; it is also clear that \cite[Claim 3.8]{c} is automatic in our case. 

We first show $a \cong \mathbf{G} \circ \mathbf{FM} (a)$ for
all $a \in D^b_{\bold{T}} (\mathfrak{E}_- )$. 
 For each $a \in D^b_{\bold{T}}(\mathfrak{E}_-)$,  we have a distinguished triangle 
 $$\longrightarrow  a \longrightarrow \mathbf{G} \circ \mathbf{FM} (a) \longrightarrow c \longrightarrow a[1] \longrightarrow. $$
In order to show $a \cong \mathbf{G} \circ \mathbf{FM} (a)$, it suffices to show that $c \cong 0$. 

Consider the collection of spanning class in Proposition \ref{spanningclass}. Let $y$ be any element in this class. Although $y$ is not supported at a single point as in \cite{c}, its support is still in one single fiber.    Denote the restriction of $y$ to  $\phi_\pm^{-1}(U_i)$ by $y_i$.

Taking ${\mathrm Hom}(y, -)$ and ${\mathrm Hom}(y_i, -)$ into the distinguished triangles, we have the following exact sequences as in proof of \cite[Proposition 3.2]{c} 
\begin{equation*}
\begin{CD}
{\mathrm Hom}^j(y,a) @>>> {\mathrm Hom}^j(y, \mathbf{G} \circ \mathbf{FM}(a)) @>>> {\mathrm Hom}^j(y,c) @>>>  {\mathrm Hom}^{j+1}(y,a) @>>>\\
@VVV @VVV @VVV @VVV  \\
{\mathrm Hom}^j(y_i,a_i) @>>> {\mathrm Hom}^j(y_i, \mathbf{G}_i \circ \mathbf{FM}_i(a)) @>>> {\mathrm Hom}^j(y_i,c_i) @>>>  {\mathrm Hom}^{j+1}(y_i,a_i) @>>>.
\end{CD}
\end{equation*}

Note that the vertical arrows are isomorphisms, ${\mathrm Hom}_{\bold{T}}(y, \mathbf{G} \circ \mathbf{FM} (a)) \cong {\mathrm Hom}_{\bold{T}}(\mathbf{FM}(y),  \mathbf{FM} (a))$ and ${\mathrm Hom}_{\bold{T}}(y_i, \mathbf{G}_i \circ \mathbf{FM}_i (a_i)) \cong {\mathrm Hom}_{\bold{T}}(\mathbf{FM}_i(y_i),  \mathbf{FM}_i (a_i))$. Since  $\mathbf{FM}_i$ are equivalences for all $i$, it follows that ${\mathrm Hom}^j_{\bold{T}}(y,c)=0$ for all $y$ from the spanning class. This shows $c \cong 0$. 

We now show 
$\mathbf{FM} \circ \mathbf{G} (b) \cong b$ for all $b \in D_{\bold{T}}^b(\mathfrak{E}_{+})$. 
Note that $\mathbf{FM}_i$ is an equivalence. It follows that 
$(\mathbf{FM}_i, G_i)$ is an adjoint pair and $\mathbf{G}_i$ is an equivalence. Using the following distinguished triangle 
$$\longrightarrow  \mathbf{FM} \circ \mathbf{G} (b) \longrightarrow b \longrightarrow c \longrightarrow  \mathbf{FM} \circ \mathbf{G} (b)[1] \longrightarrow,$$
we can show  $\mathbf{FM} \circ \mathbf{G} (b) \cong b$ as in the last part of the proof of \cite[Proposition 3.2]{c}.
\end{proof}

\end{document}